\documentclass[12pt]{amsart}
\usepackage[T1]{fontenc}
\usepackage{color}
\usepackage{url}
\usepackage{hyperref}
\usepackage{mathtools}
\mathtoolsset{showonlyrefs}
\usepackage{anysize}
\usepackage{geometry}
\marginsize{2.5cm}{2.5cm}{2.5cm}{2.5cm}
\theoremstyle{plain}
\numberwithin{equation}{section}
\newtheorem{theorem}{Theorem}[section]
\newtheorem{lemma}[theorem]{Lemma}
\newtheorem{corollary}[theorem]{Corollary}
\newtheorem{proposition}[theorem]{Proposition}
\newtheorem{remark}[theorem]{Remark}
\newcommand{\ze}{\zeta}
\newcommand{\om}{\omega}

\newcommand{\Card}{\mathsf{Card}}
\newcommand{\ii}{\mathsf{i}}
\newcommand{\ud}{\mathsf{d}}
\title[The counting problem for harmonic trinomials]{On the number of roots for harmonic trinomials}
\author{Gerardo Barrera}
\address{University of Helsinki, Department of Mathematics and Statistics.
P.O. Box 68, Pietari Kalmin katu 5, FI-00014. Helsinki, Finland.}
\email{gerardo.barreravargas@helsinki.fi}
\thanks{*Corresponding author: Gerardo Barrera.}
\author{Waldemar Barrera}
\address{Facultad de Matem\'aticas, Universidad Aut\'onoma de Yucat\'an. Anillo Perif\'erico Norte Tablaje CAT 13615, M\'erida, Yucat\'an, M\'exico.}
\email{bvargas@correo.uady.mx}
\author{Juan Pablo Navarrete}
\address{Facultad de Matem\'aticas, Universidad Aut\'onoma de Yucat\'an. Anillo Perif\'erico Norte Tablaje CAT 13615, M\'erida, Yucat\'an, M\'exico.}
\email{jp.navarrete@correo.uady.mx }
\subjclass{Primary 12D10, 30C15}
\keywords{Bohl's Theorem; Fundamental Theorem of Algebra; Harmonic Trinomials; Localization; Wilmshurst's Conjecture;  Zeros of Harmonic Trinomials}

\begin{document}
\begin{abstract}
In this manuscript we study the counting problem for harmonic trinomials of the form $a\zeta^n+b\overline{\zeta}^m+c$, where $n,m\in \mathbb{N}$, $n>m$, and $a$, $b$ and $c$ are non-zero complex numbers. 
As a consequence, we obtain the Fundamental Theorem of Algebra and the Wilmshurst conjecture for harmonic trinomials.
The proof of the counting problem relies on the Bohl method introduced in \cite{Bohl}.
\end{abstract}
\maketitle
\section{Introduction, main result and its consequences}
The computation and the understanding of the roots of a given polynomial of degree $n$ is at the heart of many problems in pure mathematics and mathematical modeling.
One of the most important and difficult problem of research in Complex Analysis is the localization of the roots of a generic polynomial of degree $n\geq 5$. In other words, 
given a generic polynomial $p$ of degree $n\geq 5$, localization refers to estimate the minimum radius $R$ such that all the roots of $p$ belong to the ball of radius $R$ centered at the origin. For further details we refer to 
\cite{Mardenbook}.
Given $r>0$ and a function $\varphi:\mathbb{C}\to \mathbb{C}$, we define the set
\begin{equation}\label{def:set}
\mathcal{Z}_{\varphi}(r):= \{\zeta\in \mathbb{C}:\quad \varphi(\zeta)=0\quad \textrm{ and } \quad|\zeta|<r\},
\end{equation}
where $|\cdot|$ denotes the complex modulus. We  denote the cardinality of a given set $X$ by $\Card(X)$ and for short we write 
$\Card_{\mathbb{Z}}(X)$ in place of $\Card(X\cap \mathbb{Z})$.
 
An important family of polynomials are the so-called trinomials of degree $n$.  
That is to say, polynomials  $h:\mathbb{C}\longrightarrow \mathbb{C}$ of the form
\begin{equation}\label{def:g}
h(\zeta):=a\ze^n+b{\ze}^m+c,
\end{equation}
where $a$, $b$ and $c$ are non-zero complex numbers, and $m$ and $n$ are positive integers satisfying $n>m$.
There are vast literature about the behavior of the roots for such polynomials.
For instance, it is well-known that the zeros are inside an annulus where the minor radius and the major radius depend only on the quantities $n$, $m$, $|a|$, $|b|$ and $|c|$, see 
\cite{Bohl,Botta,BrilleslyperSchaubroeck2018,
BrilleslyperSchaubroeck2014,Fell,Howell,Kuruklis,Melman,Nekrassoff,Szabo} and the references therein.
P. Bohl in \cite{Bohl} solves the localization and counting problem for $h$ in a ball of radius $r$ centered at the origin  according to whether the
numbers $|a|r^n$, $|b|r^m$
and $|c|$ are the lengths  of  the sides of some triangle (it may be degenerate) or not, and it reads as follows:

\begin{theorem}[Bohl's Theorem for trinomials  
\cite{Bohl}]
\hfill

\noindent
Let $r>0$ and assume that $|a|r^n$, $|b|r^m$
and $|c|$ are the side lengths of some triangle (it may be degenerate). Let $\om_1$ and $\om_2$ be the opposite angles to the sides with lengths $|a|r^n$ and $|b|r^m$, respectively. 
Then 
\[
\Card(\mathcal{Z}_{h}(r))=\Card_{\mathbb{Z}}((P-\omega(r),P+\omega(r))),
\]
where $h$ is given in \eqref{def:g},
\[
P:=\frac{n(\beta-\gamma+\pi)-m(\alpha-\gamma+\pi)}{2\pi},\quad \quad
\omega(r):=\frac{n\om_1+m\om_2}{2\pi},
\]
and $\alpha$, $\beta$, $\gamma$ are the arguments of $\textrm{a}$, $\textrm{b}$, $\textrm{c}$, respectively.
Moreover, when $|a|r^n$, $|b|r^m$
and $|c|$ are not the side lengths of any triangle, then 
\begin{equation}
\Card(\mathcal{Z}_{h}(r))=
\begin{cases}
0 & \textrm{ if }\quad |c|>|a|r^n+|b|r^m,\\
m & \textrm{ if }\quad |b|r^m>|a|r^n+|c|,\\
n & \textrm{ if }\quad |a|r^n>|b|r^m+|c|.
\end{cases}
\end{equation}
\end{theorem}
In this manuscript, we study the 
localization problem for harmonic trinomials $f:\mathbb{C}\longrightarrow \mathbb{C}$ of the form
\begin{equation}\label{eq:armonicodoble}
f(\zeta):= a\ze^n+b\overline{\ze}^m+c, 
\end{equation}
where $a,b,c$ are  non-zero complex numbers, $n,m\in \mathbb{N}$ are positive integers, and 
$\overline{\zeta}$ denotes the complex conjugate of $\zeta$. Along this manuscript we always assume that $n>m$.
It is well-known 
that~\eqref{eq:armonicodoble} has at most $n^2$ roots, see for instance 
Theorem~5 of \cite{Wilmshurst}.
Moreover, there are harmonic polynomials with exactly $n^2$ roots, see Section~2 in \cite{Bshouty} or p.~2080 of  \cite{Wilmshurst}.

Our aim is to analyze the behavior of the non-decreasing function
\begin{equation}\label{eq:card}
(0,\infty)\ni r\longmapsto \Card(\mathcal{Z}_{f}(r))\in \{0,1,\ldots,n^2\},
\end{equation}
where $f$ is given in \eqref{eq:armonicodoble} and $\mathcal{Z}_{f}(r)$ is defined in \eqref{def:set}.

In \cite{BRILLESLYPER}, the authors provide a way to count the roots for the family of harmonic trinomials when $a=1$, $c=-1$ and $b\in (0,\infty)$. Moreover, the maximum number of roots for this family is $n+2m$, see Theorem~1.1 in \cite{BRILLESLYPER}.
In \cite{Brooks}, the previous result from \cite{BRILLESLYPER} is extended to cover the case when $b\in \mathbb{C}$.
As a consequence of our main result, Theorem~\ref{th:bohlharm}, we obtain Corollary~\ref{cor:FTA} which yields that any harmonic trinomial has at most $n+2m$ roots.

We also address the Wilmshurst conjecture for the family of harmonic trinomials.  Wilmshurst's conjecture is established for harmonic polynomials of the type $h(\zeta)=p(\zeta)-\overline{q(\zeta)}$, where $p$ and $q$ are complex polynomials of degree $n$ and $m$, respectively, with $n>m$.
Wilmshurst's conjecture states that the maximum number of roots for such $h$ is bounded by $3n-2+m(m-1)$. 
In \cite{Khavinson}, the authors show that there exists a harmonic polynomial $h$ with at least $3n-2$ zeros.
As a consequence of Theorem~\ref{th:bohlharm} below, we obtain that Wilmshurst's conjecture holds true for harmonic trinomials when $n>2$ and $n>m$. 
Moreover, we obtain the existence of a family of harmonic trinomials with exactly $3n-2$ roots,
see Corollary~\ref{cor:conjec} below.

The main theorem of this manuscript is the following.
\begin{theorem}[Bohl's Theorem for harmonic trinomials]\label{th:bohlharm}
\hfill

\noindent
Let $r>0$ and assume that $|a|r^n$, $|b|r^m$
and $|c|$ are the side lengths of some triangle (it may be degenerate). Let $\om_1$ and $\om_2$ be the opposite angles to the sides with lengths $|a|r^n$ and $|b|r^m$, respectively. 
Then 
\begin{equation}\label{eq:formula}
\Card(\mathcal{Z}_{f}(r))=\Card_{\mathbb{Z}}(P_*-\om_*(0,r))+\Card_{\mathbb{Z}}(P_*+\om_*(0,r)),
\end{equation}
where $f$ is defined in \eqref{eq:armonicodoble},
\begin{equation}\label{eq:pwdoble}
P_*:=\frac{n(\beta-\gamma-\pi)+m(\alpha-\gamma-\pi)}{2\pi},
\qquad
\om_*(r):= \frac{n\om_1-m\om_2}{2\pi},
\end{equation}
\[
P_*-\om_*(0,r):=\{P_*-\om_*(u):u\in (0,r)\},
\]
\[
P_*+\om_*(0,r):=\{P_*+\om_*(u):u\in (0,r)\},
\]
and $\alpha$, $\beta$, $\gamma$ are the arguments of $\textrm{a}$, $\textrm{b}$, $\textrm{c}$, respectively.
Moreover, when $|a|r^n$, $|b|r^m$
and $|c|$ are not the side lengths of any triangle, then 
\begin{equation}\label{eq:casos123}
\Card(\mathcal{Z}_{f}(r))=
\begin{cases}
0 & \textrm{ if }\quad |c|>|a|r^n+|b|r^m,\\
m & \textrm{ if }\quad |b|r^m>|a|r^n+|c|,\\
n+2m & \textrm{ if }\quad |a|r^n>|b|r^m+|c|.
\end{cases}
\end{equation}
\end{theorem}

\begin{remark}[Extending $\om_*$]
By Lemma~\ref{eq:continuity} in Appendix~\ref{append}
we have that $\om_1$ and $\om_2$ are continuous functions of $r$. 
When  $|a|r^n$, $|b|r^m$
and $|c|$ are the side lengths of  a degenerate triangle, we define  
\begin{equation}
\omega_*(r)=
\begin{cases}
0   & \textrm{ if } |c|=|a|r^n+|b|r^m,\\
\frac{n}{2}   & \textrm{ if } |a|r^n=|b|r^m+|c|,\\
-\frac{m}{2}   & \textrm{ if } |b|r^m=|a|r^n+|c| , 
\end{cases}
\end{equation}
and extend it continuously as follows
\begin{equation}
\omega_*(r)=
\begin{cases}
0   & \textrm{ if } |c|\geq |a|r^n+|b|r^m,\\
\frac{n}{2}   & \textrm{ if } |a|r^n\geq |b|r^m+|c|,\\
-\frac{m}{2}   & \textrm{ if } |b|r^m\geq |a|r^n+|c|.
\end{cases}
\end{equation}
\end{remark}
Theorem~\ref{th:bohlharm} yields that the number of roots for any harmonic trinomial is at most $n+2m$.

\begin{corollary}[Fundamental Theorem of Algebra for harmonic trinomials]\label{cor:FTA}
\hfill

\noindent
Any harmonic 
trinomial~\eqref{eq:armonicodoble} has at most $n+2m$ roots.
Moreover, there exists a family of  harmonic trinomials $(g_{r})_{r>0}$ with exactly $n+2m$ roots.
\end{corollary}

\begin{remark}
For $n$ and $m$ co-primes, the case $a=1, c=-1$ and  $b>0$ is covered in Theorem~1.1 of \cite{BRILLESLYPER},
while the case  $a=1, c=-1$ and  $b\in \mathbb{C}$ is covered in Theorem~1.1 of \cite{Brooks}.
\end{remark}

\begin{proof}[Proof of Corollary~\ref{cor:FTA}]
Let $\mathfrak{a}$ be the unique positive root of the trinomial $A(r)=|a|r^n-|b|r^m-|c|$. For $r>\mathfrak{a}$ we have $|a|r^n>|b|r^m+|c|$ and 
hence~\eqref{eq:casos123} yields the result.
\end{proof}

As a consequence of Corollary~\ref{cor:FTA} we obtain Wilmshurst's conjecture for harmonic trinomials. 
In general,
Wilmshurst's conjecture does not hold for harmonic polynomials, see \cite{Hauenstein,Lee}.

\begin{corollary}[Wilmshurst's conjecture for harmonic trinomials]\label{cor:conjec}
\hfill

\noindent
Given $n\in \mathbb{N}\setminus\{1\}$ and $m\in \{1,\ldots,n-1\}$ it follows that
any harmonic trinomial possesses at most $3n-2$ roots.
Moreover, there exists a family of  harmonic trinomials 
$(g_{r})_{r>0}$ with exactly $3n-2$ roots.
In addition, for $m=1$, the number of roots is at most  $n+2$.
\end{corollary}

\begin{remark}
Using complex dynamics and the argument principle, the authors in \cite{KhavinsonSwiatek} prove  the conjecture of
Sheil--Small and Wilmshurst that establishes 
\begin{equation}\label{eq:eq1}
\Card \{\zeta\in \mathbb{C} : p(\zeta) - \overline{q(\zeta)} = 0\} \leq 3n - 2,
\end{equation}
where the polynomial $p$ has degree $n>1$ and $q(\zeta)=\zeta$ for all $\zeta \in \mathbb{C}$.
In the end of p.~413, they also stress that the monomial case: $1<m<n$ and $q(\zeta)=\zeta^m$ for all $\zeta\in \mathbb{C}$
requires a deep analysis of the dynamics of a root map on the Riemann surface.
Later on, in \cite{Geyer} the author proves that the upper bound $3n-2$ is sharp by proving the existence of a complex polynomial $p$ of
degree $n$ such that 
\[
\Card \{\zeta\in \mathbb{C} : p(\zeta) - \overline{q(\zeta)} = 0\}= 3n - 2.
\]
We stress that
Corollary~\ref{cor:conjec} 
yields~\eqref{eq:eq1} for the particular case: $n>2$, $n>m$ and $p(\zeta)=\zeta^n$, $q(\zeta)=\zeta^m$ for all $\zeta\in \mathbb{C}$. Moreover, it also gives
the existence of a family of harmonic trinomials
with exactly $3n-2$ roots.
\end{remark}

\begin{proof}[Proof of Corollary~\ref{cor:conjec}]
By Corollary~\ref{cor:FTA} we have that any harmonic trinomial possesses at most $n+2m$ roots.
Since $m\leq n-1$, we have
\[
n+2m\leq n+2(n-1)=3n-2.
\]
The equality holds when $m=n-1$. 
Corollary~\ref{cor:FTA} yields the existence of
a family of harmonic trinomials $(g_{r})_{r>0}$ with exactly $3n-2$ roots.
\end{proof}

The manuscript is organized as follows. 
In Section~\ref{sec:bohlmethod} we develop  Bohl's method for harmonic trinomials.
Next, in Section~\ref{sec:proof} we provide the proof of  Theorem~\ref{th:bohlharm}.
Finally, in Appendix~\ref{append}
we state auxiliary results that we use throughout the manuscript.

\section{The Bohl method}\label{sec:bohlmethod}
In this section, we develop Bohl's method \cite{Bohl} for harmonic trinomials. First, Subsection~\ref{subsec:coprime} allows us to reduce the proof of Theorem~\ref{th:bohlharm} to the co-prime case. Next, Subsection~\ref{subsec:positive} yields that the original harmonic trinomial~\eqref{eq:armonicodoble} can be transformed to a simplified harmonic trinomial such that the coefficient $c$ is positive and both have the same roots. This allows us to relate the roots of the simplified harmonic trinomial according to whether the numbers $|a|r^n$, $|b|r^m$ and $c$ (for some $r>0$) are the side lengths of some triangle. 
Finally, Subsection~\ref{subsec:region} yields the region of zeros.

\subsection{Reduction to the co-prime case}\label{subsec:coprime}
In this subsection we stress that it is enough to show Theorem~\ref{th:bohlharm} in the case that $n$ and $m$ are co-prime numbers. The general case can be deduced from the co-prime case as follows.

\begin{corollary}[Reduction to the co-prime case]\label{cor:bohlharmgeneral}
\hfill

\noindent
Assume that Theorem~\ref{th:bohlharm} holds true for harmonic trinomials with co-prime exponents.
Let $n,m\in \mathbb{N}$ such that $n>m$,
then Theorem~\ref{th:bohlharm} holds true for harmonic trinomials with exponents $n$ and $m$. More precisely, 
let $r>0$ and assume that $|a|r^n$, $|b|r^m$
and $|c|$ are the side lengths of some triangle (it can be degenerate). Let $\om_1$ and $\om_2$ be the opposite angles to the sides with lengths $|a|r^n$ and $|b|r^m$, respectively. 
Then 
\[
\Card(\mathcal{Z}_{f}(r))=\Card_{\mathbb{Z}}(P_*-\om_*(0,r))+\Card_{\mathbb{Z}}(P_*+\om_*(0,r)),
\]
where 
\begin{equation}
P_*:=\frac{n(\beta-\gamma-\pi)+m(\alpha-\gamma-\pi)}{2\pi},
\qquad
\om_*(r):= \frac{n\om_1-m\om_2}{2\pi},
\end{equation}
and $\alpha$, $\beta$, $\gamma$ are the arguments of $\textrm{a}$, $\textrm{b}$, $\textrm{c}$, respectively.
Moreover, when $|a|r^n$, $|b|r^m$
and $|c|$ are not side lengths of any triangle, then 
\begin{equation}
\Card(\mathcal{Z}_{f}(r))=
\begin{cases}
0 & \textrm{ if }\quad |c|>|a|r^n+|b|r^m,\\
m & \textrm{ if }\quad |b|r^m>|a|r^n+|c|,\\
n+2m & \textrm{ if }\quad |a|r^n>|b|r^m+|c|.
\end{cases}
\end{equation}
\end{corollary}

\begin{proof}
Let $d=\gcd(n,m)$ be the greatest common divisor of $n$ and $m$ and set $\widetilde{n}:=n/d$, $\widetilde{m}:=m/d$. We observe that $\gcd(\widetilde{n},\widetilde{m})=1$.
We consider the harmonic trinomial equation  
\[
h(\zeta)= a\zeta^{\widetilde{n}}+b\overline{\zeta}^{\widetilde{m}}+c=0.
\] 
By hypothesis Theorem~\ref{th:bohlharm} holds true for the exponents $\widetilde{n}$ and $\widetilde{m}$. That is,
\begin{equation}\label{eq:mdc}
\Card(\mathcal{Z}_{h}(s))=\Card_{\mathbb{Z}}(\widetilde{P}_*-\widetilde{\om}_{*}(0,s))+\Card_{\mathbb{Z}}(\widetilde{P}_*+\widetilde{\om}_*(0,s))
\end{equation}
for any $s\geq 0$,
where 
\[
\widetilde{P}_{*}:=\frac{\widetilde{n}(\beta-\gamma-\pi)+\widetilde{m}(\alpha-\gamma-\pi)}{2\pi}\quad \textrm{ and } \quad \widetilde{\om}_{*}(s):= \frac{\widetilde{n}\widetilde{\om}_1-\widetilde{m}\widetilde{\om}_2}{2\pi}.
\]
Recall that $\widetilde{\om}_1$ and $\widetilde{\om}_2$ are the 
interior angles of the triangle 
 with side lengths $|a|s^{\widetilde{n}}$, $|b|s^{\widetilde{m}}$ and $|c|$
opposite to the side lengths $|a|s^{\widetilde{n}}$ and $|b|s^{\widetilde{m}}$, respectively. We note that for  $s= r^d$ we have 
\[ 
P_*=d\widetilde{P}_{*} 
\quad \textrm{ and }\quad\om_{*}(r)= d\widetilde{\om}_{*}(r^d).
 \] 
Hence~\eqref{eq:mdc} yields
 $\Card(\mathcal{Z}_f(r))= d\cdot \Card(\mathcal{Z}_h(r^d))$, which implies the statement.
\end{proof}

\subsection{Reduction to the case $c>0$}\label{subsec:positive}
As a consequence of Corollary~\ref{cor:bohlharmgeneral},
without loss of generality, from here to the end of this manuscript, we always assume that $n$ and $m$ are co-prime numbers.
In this subsection we start showing that the coefficient $c$ can be assumed to be positive as the following lemma states.
\begin{lemma}\label{lem:equiv}
Let $a,b,c$ be any complex numbers and consider the harmonic trinomial given by
$f(\zeta)= a\zeta^n+b\overline{\zeta}^m +c$, $\zeta\in \mathbb{C}$. Define the harmonic trinomial by $\widetilde{f}(\zeta)=a_*\zeta^n+b_*\overline{\zeta}^m+|c|$, $\zeta\in \mathbb{C}$, where $a_*= ae^{-\ii\gamma}$, $b_*= be^{-\ii\gamma}$ and $\gamma=\arg(c)$. Then $\mathcal{Z}_{f}(r)= \mathcal{Z}_{\widetilde{f}}(r)$ for any $r>0$.
\end{lemma}
\begin{proof}
We claim that  a complex number $\zeta$ is a root of $f$ if and only if $\zeta$ is  a root of $\widetilde{f}$. Indeed, by definition we have
$0=f(\zeta)=
a\zeta^n+b\overline{\zeta}^m +c$,
which is equivalent to $ae^{-\ii\gamma}\zeta^n+be^{-\ii\gamma}\overline{\zeta}^m +ce^{-\ii\gamma}=0$.
The latter reads as $\widetilde{f}(\zeta)=a_*\zeta^n+b_* \overline{\zeta}^m +|c|=0$.
This completes the proof of 
Lemma~\ref{lem:equiv}.
\end{proof}
By Lemma~\ref{lem:equiv} we can assume that $a$ and $b$ are any non-zero complex numbers and 
 $c>0$. Then for convenience and in a conscious abuse of notation we set
\begin{equation}\label{eq:hec}
f(\zeta)=a\zeta^n+b\overline{\zeta}^m+c, \qquad \zeta \in \mathbb{C}.
\end{equation} 
Let $r>0$ and assume that $|a|r^n$, $|b|r^m$
and $c$ are   the side lengths of some triangle (it may be degenerate).
Let $\om_1$ and $\om_2$ be the opposite angles to the sides with lengths $|a|r^n$ and $|b|r^m$, respectively. 
In this setting, the pivot terms $P_*$ and $\omega_*(r)$ defined in~\eqref{eq:pwdoble}  are given by
\begin{equation}\label{eq:pivotaltes}
P_*=\frac{n(\beta-\pi)+m(\alpha-\pi)}{2\pi}
\quad \textrm{ and }\quad
\omega_*(r)= \frac{n\om_1-m\om_2}{2\pi}.
\end{equation}
We point out that $P_*$ only depends on $n$, $m$, $\arg(a)$ and $\arg(b)$. Whereas $\omega_*(r)$ only depends on $|a|r^n$, $|b|r^m$ and $c$.

The following propositions allow us to relate the modulus of the roots for the harmonic trinomial equation~\eqref{eq:hec} with the existence of some triangle and arithmetic properties of the pivots \eqref{eq:pivotaltes}.
\begin{proposition}\label{prop:pomega}
Let $\zeta_0$ be any root of the harmonic trinomial equation~\eqref{eq:hec} and set $r=|\zeta_0|$.
Assume that the numbers $|a|r^n$, $|b|r^m$ and $c$ are  the side lengths of some triangle $\Delta$. Then  $P_*-\om_*(r)\in \mathbb{Z}$ or $P_*+\om_*(r)\in \mathbb{Z}$, where $P_*$ and $\om_*(r)$ are given in \eqref{eq:pivotaltes}.
\end{proposition}

\begin{proof}
The proof is a slight modification of the ideas given in \cite{Bohl}, p.~559.
We write the complex numbers $a$, $b$ and $\zeta_0$ in polar form. That is to say, we write 
$a=|a|e^{\ii \alpha}$, $b=|b|e^{\ii \beta}$ and
$\zeta_0=re^{\ii \phi}$ for some $r>0$ and $\alpha,\beta,\phi\in [0,2\pi)$.
Since $\zeta_0$ is a root of~\eqref{eq:hec},
we have
\begin{equation}\label{eq:trigoeq}
\begin{split}
0=a\zeta_0^n+b\overline{\zeta}_0^m+c&=
|a|e^{\ii \alpha}r^n e^{\ii n\phi}+|b|e^{\ii \beta}r^m e^{-\ii m\phi}+c\\
&=|a|r^ne^{\ii(\alpha +n\phi)} +|b|r^m e^{\ii(\beta-m\phi)}+c.
\end{split}
\end{equation}
By hypothesis the numbers $|a|r^n$, $|b|r^m$ and $c$ are the side lengths of $\Delta$, and let $\om_1$ and $\om_2$ be the interior angles of $\Delta$ opposite to the sides with lengths $|a|r^n$ and $|b|r^n$, respectively.
Then \eqref{eq:trigoeq} with the help of Lemma~\ref{prop:congru} in Appendix~\ref{append} yields the following two cases:
\begin{itemize}
\item[$\bullet$] Case I: The following relations holds true
\begin{equation}\label{eq:rel1}
\alpha+n\phi\equiv \pi-\om_2\quad\textrm{ and }\quad \beta-m\phi\equiv \pi+\om_1.
\end{equation}
\item[$\bullet$] Case II: The following relations holds true
\begin{equation}\label{eq:rel2}
\alpha+n\phi\equiv -\pi+\om_2\quad
\textrm{ and }\quad
 \beta-m\phi\equiv -\pi-\om_1.
\end{equation}
\end{itemize}
In the sequel, we analyze Case I.
By definition of the symbol 
$\equiv$, relation~\eqref{eq:rel1} reads as follows:
$
\alpha= \pi-\om_2-n\phi +2\pi k_1$
and $\beta= \pi+\om_1+m\phi +2\pi k_2$ 
for some $k_1,k_2 \in \mathbb{Z}$. Straightforward computations yield
\begin{equation}
P_*-\omega_*(r)= \frac{n\beta+m\alpha-(n+m)\pi-n\om_1+m\om_2}{2\pi}=k_2n+k_1m\in \mathbb{Z}.
\end{equation}
The Case II is completely analogous.
This completes the proof of Proposition~\ref{prop:pomega}.
\end{proof}
Roughly speaking, the next proposition is the converse of Proposition~\ref{prop:pomega}.
\begin{proposition}\label{prop:vuelta}
Let $r>0$ and assume that $|a|r^n$, $|b|r^m$ and $c$ are the side lengths of some triangle $\Delta$. 
Let $P_*$ and $\om_*(r)$ be the pivots defined in \eqref{eq:pivotaltes}.
If $P_*-\om_*(r)$ or $P_*+\om_*(r)$ are integers, then the harmonic trinomial equation~\eqref{eq:hec} has at least one root with modulus $r$.
\end{proposition}

\begin{proof}
The proof is a slight modification of the ideas given in \cite{Bohl}, p.~559.
The assumption $P_*-\om_*(r)$ or $P_*+\om(r)$ are integers reads as follows:
\begin{equation}\label{equ:ec1}
n(\beta-\pi)+m(\alpha-\pi)+\epsilon(n\om_1-m\om_2)= 2\pi \kappa,
\end{equation}
where $\epsilon\in \{-1,1\}$ and $\kappa:=\kappa(\epsilon)\in \mathbb{Z}$. 
Since the natural numbers $n$ and $m$ are co-primes, there exist integers $N$ and $M$ satisfying
\begin{equation}\label{equ:ec2}
nN-mM=\kappa.
\end{equation}
By~\eqref{equ:ec1} and~\eqref{equ:ec2} we obtain
\begin{equation}\label{eq:no3}
 n(\beta -\pi +\epsilon \om_1-2\pi N+m y)+m(\alpha-\pi-\epsilon \om_2+2\pi M-ny)=0
\end{equation} 
for any $y\in \mathbb{C}$.
Now, we choose the unique $z\in \mathbb{R}$ such that $\beta -\pi +\epsilon \om_1-2\pi N+mz=0$. This yields
\begin{equation}\label{ec3}
\om_1\equiv \epsilon(\pi-\beta-mz).
\end{equation}
By~\eqref{eq:no3} we have
$\alpha-\pi-\epsilon \om_2+2\pi M-nz=0$, which implies
\begin{equation}\label{ec4}
\om_2\equiv \epsilon(\alpha-\pi -nz).
\end{equation}
Since $|a|r^n$, $|b|r^m$ and $|c|$ are the side lengths of $\Delta$,
Lemma~\ref{triangulo} in Appendix~\ref{append} with the help 
of~\eqref{ec3}  and~\eqref{ec4} implies
\begin{equation}\label{eq:zero1}
0=|a|r^n e^{-\ii \om_2}+|b|r^me^{\ii \om_1}-c
=
|a|r^n e^{-\ii \epsilon(\alpha-\pi-nz)}+|b|r^me^{\ii\epsilon(\pi-\beta-mz)}-c.
\end{equation}
Recall that $\alpha$ and $\beta$ are the arguments of $a$ and $b$, respectively. Then we write
$a=|a|e^{\ii \alpha}$ and $b=|b|e^{\ii \beta}$.
In the sequel, we analyze the case 
$\epsilon= 1$. Relation~\eqref{eq:zero1} reads as follows: 
\begin{equation}
\begin{split}
0&=|a|r^n e^{-\ii(\alpha-\pi-nz)}+|b|r^me^{\ii(\pi-\beta-mz)}-c\\
&=|a|r^n e^{-\ii\alpha}e^{\ii\pi}e^{\ii nz}+|b|r^m e^{\ii \pi}e^{-\ii\beta}e^{-\ii mz}-c.
\end{split}
\end{equation}
Then we have 
$
0=ar^n e^{-\ii nz}+{b}r^me^{\ii mz}+c=a\zeta^n_1+{b}\overline{\zeta}^m_1+c,
$
where $\zeta_1:=re^{-\ii z}$.  Hence  $\zeta_1$ is a root of~\eqref{eq:hec}.

We continue with the case $\epsilon=-1$. Relation~\eqref{eq:zero1} implies
\[
0=|a|r^n e^{\ii(\alpha-\pi-nz)}+|b|r^me^{-\ii(\pi-\beta-mz)}-c=
|a|r^n e^{\ii\alpha}e^{-\ii\pi}e^{-\ii nz}+|b|r^me^{-\ii\pi}e^{\ii\beta} e^{\ii mz}-c.
\]
The preceding equation yields
$
0=ar^n e^{-\ii nz}+br^m e^{\ii mz}+c=a{\zeta}_{-1}^n +b\overline{\zeta}^m_{-1}+c,
$
where
$\zeta_{-1}:=r e^{-\ii z}$. Hence, $\zeta_{-1}$ is a root of~\eqref{eq:hec}.
This finishes the proof of 
Proposition~\ref{prop:vuelta}.
\end{proof}

In the following proposition, we analyze arithmetic properties of the extremes $P_*-\omega_*(r)$ and $P_*+\omega_*(r)$. In particular, in the generic case $n\beta+m\alpha\not=k\pi$ for some $k\in \mathbb{N}$, we obtain that $P_*-\omega_*(r)$ and $P_*+\omega_*(r)$  cannot be both integers.
\begin{proposition}\label{eq:prop}
Let $P_*$ and $\om_*(r)$ be the pivots defined in \eqref{eq:pivotaltes}.
The  following statements are valid.
\begin{enumerate}
\item[(i)] If $n\beta +m\alpha$ is an integer multiple of $\pi$ and $P_*-\om_*(r)$ is an integer, then $P_*+\om_*(r)$ is an integer.
\item[(ii)] If $n\beta +m\alpha$ is an integer multiple of $\pi$ and $P_*+\om_*(r)$ is an integer, then $P_*-\om_*(r)$ is an integer.
\item[(iii)] If $n\beta +m\alpha$ is not an integer multiple of $\pi$, then  $P_*-\om_*(r)$ and $P_*+\om_*(r)$  cannot be both integers.
\item[(iv)] $n\beta +m\alpha$ is an integer multiple of $\pi$ if and only if $2P_*$ is an integer.
\end{enumerate}
\end{proposition}

\begin{proof}
We start with the proof of Item (i).
If $n\beta +m\alpha= \pi k_1$ for some $k\in \mathbb{Z}$ and $P_*-\om(r)=k_2  \in\mathbb{Z}$, then  straightforward computations yield
\begin{equation}
\begin{split}
P_*+\om_*(r)&= \frac{\pi (k_1-n-m)}{2\pi}+\frac{n\om_1-m\om_2}{2\pi}= \frac{k_1-n-m}{2}+\frac{k-n-m}{2}-k_2\\
&= k_1-k_2-n-m\in \mathbb{Z}.
\end{split}
\end{equation}
This finishes the proof of Item (i).
The proof of Item (ii) is completely analogous and we omit it. 
Now, we prove Item (iii) by a contradiction argument.
Assume that $P_*-\om_*(r)$ and $P_*+\om_*(r)$ are integers. Then $2P_*$ is an integer and satisfies
\begin{equation}\label{eq:rel56}
2P_*+(n+m)=\frac{n\beta+m\alpha}{\pi},
\end{equation}
which is a contradiction due to the
left-hand side of the preceding inequality is an integer, whereas the
 right-hand side is not an integer.
This completes the proof of Item (iii).
Finally,  Item (iv) follows directly from \eqref{eq:rel56}.
\end{proof}

\subsection{The region with roots.}\label{subsec:region}
In this subsection, we find the region  where the modulus of the roots belong.
Recall that we assume that $c>0$. However, for convenience, in this subsection we write $|c|$ instead of $c$
since the statements of this subsection hold true for any $c\in \mathbb{C}\setminus\{0\}$.
\begin{lemma}\label{lemTec2-1}
Let $r>0$ and suppose that $|a|r^n$, $|b|r^m$ and $|c|$ are not the side lengths of any  triangle.
Then there is no solution of modulus $r$ of  the harmonic trinomial equation~\eqref{eq:hec}.
\end{lemma}

\begin{proof}
We use a contradiction argument.
Suppose that $\zeta_0$ is  a root of modulus $r$ of the equation~\eqref{eq:armonicodoble}. 
Assume that $|c|>|a|r^n+|b|r^m$. Since $\zeta_0$ is a root 
of~\eqref{eq:armonicodoble}, we have 
\[
|c|=|-a\zeta^n_0 -b\overline{\zeta}^m_0|=|a\zeta^n_0 +b\overline{\zeta}_0^m|\leq |a|r^n+|b|r^m<|c|,
\]
which is a contradiction.
A similar reasoning applies for the cases $|a|r^n>|b|r^m+|c|$ and  $|b|r^m>|a|r^n+|c|$.
\end{proof}

The preceding lemma motivates the following definition. We set
\begin{equation}\label{eq:defdom}
\mathbb{T}:=\left\{r\in (0,\infty):
|a|r^n,|b|r^m,|c|\,\,\textrm{ are the side lengths of some triangle}
\right\}.
\end{equation}
In~\eqref{eq:defdom} we have excluded the case of degenerate triangles. They are studied separately.
We start analyzing the precise shape of $\mathbb{T}$.
\begin{theorem}[The shape of $\mathbb{T}$]\label{th:ABC}
For any  $r\in [0,\infty)$ let
\begin{equation}\label{eq:ABC}
\begin{split}
&A(r)=|a|r^n-|b|r^m-|c|,\quad B(r)= -|a|r^n+|b|r^m-|c|\\
&\textrm{and }\quad C(r)= -|a|r^n-|b|r^m+|c|.
\end{split}
\end{equation}  
Let $B_{n,m}$ be the maximum of the function $B$ over $[0,\infty)$. 
Let $\mathfrak{a}$ and $\mathfrak{c}$ be the unique  positive  roots of $A$ and  $C$, respectively.
Then the following holds true.
\begin{itemize}
\item[(I)] If $B_{n,m}<0$, it follows that $\mathbb{T}=(\mathfrak{c},\mathfrak{a})$. 
\item[(II)]  If $B_{n,m}=0$, it follows that $\mathbb{T}=(\mathfrak{c}, \mathfrak{b})\cup(\mathfrak{b},\mathfrak{a})$, where $\mathfrak{b}$ is the unique positive root of $B$.
\item[(III)] If $B_{n,m}>0$, then  $B$ has exactly  two positives roots $\mathfrak{b}_1<\mathfrak{b}_2$ and $\mathbb{T}=
(\mathfrak{c}, \mathfrak{b}_1)\cup(\mathfrak{b}_2,\mathfrak{a})$.
\end{itemize}
\end{theorem}

\begin{proof}
The proof of Theorem~\ref{th:ABC} is given in \cite{Bohl}. However, for reader's convenience, we sketch it here. By the Rule of Signs of Descartes, the trinomial $A$ has precisely one positive root, that we denote by $\mathfrak{a}$. Similarly, $C$ has precisely one positive root  that we denote $\mathfrak{c}$. A more delicate analysis shows that $B$ could have zero, one or two positive roots. In fact, the last statement depends whether to the maximum $B_{n,m}$ of $B$ is negative, zero or positive, respectively.

We start by showing that $\mathfrak{c}<\mathfrak{a}$. 
We note that the function $A$ satisfies the following:
\begin{itemize}
\item[(a)] The number $r_0= (\frac{m|b|}{n|a|})^{\frac{1}{n-m}}$ is the only positive critical point of the derivative $\frac{\ud A}{\ud r}$. Moreover, the derivative of $A$ with respect to $r$ is negative for $0<r<r_0$ and positive for $r>r_0$.
\item[(b)]  
$A(r)$ is negative  for $0<r<\mathfrak{a}$, and $A(r)$ is positive for $r>\mathfrak{a}$.
\end{itemize}
The function $C$ has the following properties:
\begin{itemize}
\item[(c)]
$C(r)$ is positive for  $0<r<\mathfrak{c}$, and $C$
is negative for $r>\mathfrak{c}$.
\end{itemize}
We also note  that $(A+C)(\mathfrak{c})= -2|b|\mathfrak{c}^m<0$. Since $C(\mathfrak{c})=0$, we have  $A(\mathfrak{c})<0$ which implies $\mathfrak{a}<\mathfrak{c}$. 
Moreover, by Item (b) and Item (c) we deduce 
that $\mathbb{T}\subset (\mathfrak{c},\mathfrak{a})$.

We point out that Lemma~\ref{lemTec2-1} guarantees that the modules of the roots $r$ of an harmonic trinomial lie in the open subset of $(0,\infty)$ determined by the inequalities $A(r)\leq 0$, $B(r)\leq 0$ and $C(r)\leq 0$. 
We observe the following properties for the trinomial $B$:
\begin{itemize}
\item[(d)] The number $r_0$ is the unique critical point of the derivative $\frac{\ud B}{\ud r}$. Moreover, such derivative is positive for $0<r<r_0$ and negative for $r>r_0$.
\item[(e)] $B_{n,m}=B(r_0)\in \mathbb{R}$.
\end{itemize}
First, we prove Item (I). We assume that $B_{n,m}<0$ let $r\in (\mathfrak{c},\mathfrak{a})$ be fixed. 
Since $B_{n,m}<0$, we have that $B(r)<0$. 
Hence, by~\eqref{eq:defdom} we obtain $(\mathfrak{c},\mathfrak{a})\subset \mathbb{T}$.

Now, we prove Item (II). Since $B_{n,m}=0$,  we have that $r_0=\mathfrak{b}$ and hence $B(r)<0$ for any $r\in (0,r_0)\cup (r_0,\infty)$. Note that 
$(B+A)(r_0)<0$ and $(B+C)(r_0)<0$. By Item (b) and Item (c) we obtain $r_0=\mathfrak{b}\in (\mathfrak{c},\mathfrak{a})$ and then $\mathbb{T}=(\mathfrak{c}, \mathfrak{b})\cup(\mathfrak{b},\mathfrak{a})$.

Finally, we prove Item (III). Since $B_{n,m}>0$, we have that $B$ has exactly two positive roots $\mathfrak{b}_1$ and $\mathfrak{b}_2$ with $\mathfrak{b}_1<\mathfrak{b}_2$ . Moreover, $B(r)\leq 0$ for $r\in (0,\mathfrak{b}_1]\cup [\mathfrak{b}_2,\infty)$. An analogous reasoning used in the proof of Item (II) yields
$\mathbb{T}=
(\mathfrak{c}, \mathfrak{b}_1)\cup(\mathfrak{b}_2,\mathfrak{a})$.
\end{proof}

In the sequel, we analyze the boundary of 
$\mathbb{T}$. The proofs of the following three lemmas are straightforward and  we omit them.
\begin{lemma}[The boundary cases]\label{lem:boundedcases}
Let keep the notation introduced in Theorem~\ref{th:ABC}. Then the following is valid.
\begin{itemize}
\item[(i)]  $r= \mathfrak{c}$ if and only if $|c|= |a|r^n+|b|r^m$.
\item[(ii)] $r= \mathfrak{a}$ if and only if $|a|r^n= |b|r^m+|c|$.
\item[(iii)]  Assume $B_{n,m}=0$. Then $r= \mathfrak{b}$, if and only if $|b|r^m= |a|r^n+|c|$.
\item[(iv)]  Assume $B_{n,m}>0$. Then $r\in \{\mathfrak{b}_1,\mathfrak{b}_2\}$ if and only if $|b|r^m = |a|r^n+|c|$.
\end{itemize}
\end{lemma}

In what follows we analyze $\overline{\mathbb{T}}^\textnormal{c}$.
\begin{lemma}[The complement of $\overline{\mathbb{T}}$]\label{lem:complemento}
Let keep the notation introduced in Theorem~\ref{th:ABC}. Then the following is valid.
\begin{itemize}
\item[(i)]  $r\in(0,\mathfrak{c})$ if and only if $|c|>|a|r^n+|b|r^m$.
\item[(ii)]  $r\in(\mathfrak{a},\infty)$ if and only if $|a|r^m>|b|r^m+|c|$.
\item[(iii)]  $r\in (\mathfrak{b}_1,\mathfrak{b}_2)$ and $B(r)\neq B_{n,m}$ if and only if $|b|r^m>|a|r^n+|c|$.
\end{itemize}
\end{lemma}

\begin{lemma}\label{lem:notriangle}
Let keep the notation introduced in Theorem~\ref{th:ABC}.
Then it follows that
\begin{itemize}
\item[(i)] If $r\in (0,\mathfrak{c}]$, then $\om_*(r)=0$.
\item[(ii)] If $r\in [\mathfrak{a},\infty)$, then $\om_*(r)=\frac{n}{2}$.
\item[(iii)] If $r=\mathfrak{b}$ or $r\in [\mathfrak{b}_1,\mathfrak{b}_2]$, then 
 $\om_*(r)= -\frac{m}{2}$.
\end{itemize}
\end{lemma}
Now, we show that in the generic case (iii) of Proposition~\ref{eq:prop} and in the boundary of $\mathbb{T}$, there is no roots for the harmonic trinomial. This is precisely stated in the following lemma.
\begin{lemma}\label{eq:casogenerico09} 
Assume that $n\beta + m\alpha$ is not an integer multiple of $\pi$ and for some $r>0$, the numbers $|a|r^n$, $|b|r^m$  and $|c|$ are the side lengths of some degenerate triangle. Then there is no  root of the harmonic 
trinomial~\eqref{eq:hec} of modulus $r$.
\end{lemma}
\begin{proof}
We use a contradiction argument.
Let $r= \mathfrak{c}$, $\zeta_0= re^{i\theta}$ and assume that $\zeta_0$ is a root of harmonic trinomial  equation. Then we have
\begin{equation}\label{eq:Tgh}
|a|r^n e^{\ii(n\theta+\alpha)}+|b|r^m e^{\ii(-m\theta +\beta)}+|c|=0.
\end{equation}
Lemma~\ref{lem:boundedcases} implies $|c|= |a|r^n+|b|r^m$. By~\eqref{eq:Tgh} and 
Lemma~\ref{prop:congru} in Appendix~\ref{append} we obtain $\alpha+n\theta\equiv \pi $ and $\beta -m\theta\equiv \pi$. Hence, 
$\alpha= -n\theta +\pi +2\pi k_1$ and $\beta=\pi +m\theta +2\pi k_2$ 
for some $k_1, k_2\in \mathbb{Z}$.
By straightforward computations we have
$n\beta+m\alpha=  \pi(n+m +2k_2n+ 2k_1m)$
which yields a contradiction.
A similar reasoning applies for the cases 
$|a|r^n=|b|r^m+|c|$ and  $|b|r^m=|a|r^n+|c|$.
\end{proof}
Now, we prove that the function $\omega_*$ is piece-wise monotone.
\begin{lemma}\label{lem:omegaderivada}
The derivative of the function $\om_*:\mathbb{T}\rightarrow \mathbb{R}$ satisfies the following:
\begin{equation}\label{eq:omegaprime}
\begin{split}
&\frac{\ud}{\ud r}\om_*(r)<0\quad \textrm{ for }\quad r\in (\mathfrak{c},r_0)\cap \mathbb{T},\\
&
\frac{\ud}{\ud r}\om_*(r)>0\quad \textrm{ for }\quad r\in (r_0,\mathfrak{a})\cap \mathbb{T},
\end{split}
\end{equation}
where $r_0= (\frac{m|b|}{n|a|})^{\frac{1}{n-m}}$.
\end{lemma}

\begin{proof}
We compute the derivative of $\om_*$ in the  domain $\mathbb{T}$. For any $r\in \mathbb{T}$,
the law of cosines yields
\begin{equation}\label{eq:GHT}
\begin{split}
|a|^2 r^{2n}&=|b|^2r^{2m}+|c|^2 -2|b||c|r^m \cos(\om_1(r)),\\
|b|^2 r^{2m}&=|a|^2 r^{2n}+|c|^2 -2|a||c|r^n \cos(\om_2(r)).
\end{split}
\end{equation}
For shorthand we write $\om_1$ and $\om_2$ instead of $\om_1(r)$ and $\om_2(r)$.
By  straightforward computations we obtain
\begin{equation}\label{eq:hoi}
\begin{split}
\frac{\ud \om_1}{\ud r}
&=\frac{n|a|^2 r^{2n-1}-m|b|^2 r^{2m-1}+m|b||c|r^{m-1}\cos(\om_1)}{|b||c|r^m \sin(\om_1)},\\
\frac{\ud \om_2}{\ud r}&=\frac{m|b|^2 r^{2m-1}-n|a|^2 r^{2n-1}+n|a||c|r^{n-1}\cos(\om_2)}{|a||c|r^n \sin(\om_2)}.
\end{split}
\end{equation}
The law of sines implies $\delta:=|b||c|r^m\sin(\om_1)=|a||c|r^n\sin(\om_2)>0$.
Hence, \eqref{eq:hoi} yields
\begin{equation}
\begin{split}
2\pi \frac{\ud \om_*}{\ud r}&=
n\frac{\ud \om_1}{\ud r}-m\frac{\ud \om_2}{\ud r}
=\frac{(n|a|r^n-m|b|r^m)(n|a|r^n+m|b|r^m)}{2r\delta},
\end{split}
\end{equation}
which has a unique critical point at $r_0= (\frac{m|b|}{n|a|})^{\frac{1}{n-m}}$.
This easily implies~\eqref{eq:omegaprime}.
\end{proof}

\section{Proof of Theorem~\ref{th:bohlharm}}\label{sec:proof}
In this section, we stress the fact that Theorem~\ref{th:bohlharm} is just a consequence of what we have already proved in Section~\ref{sec:bohlmethod}.
\begin{proof}[Proof of 
Theorem~\ref{th:bohlharm}]
We use the same notation introduced in Theorem~\ref{th:ABC}. 
Now, we give with the proof of Case~(I).
We recall that
$\mathbb{T}=(\mathfrak{c},\mathfrak{a})$ and $\mathfrak{b}=\mathfrak{r}_0= (\frac{m|b|}{n|a|})^{\frac{1}{n-m}}$.
By Lemma~\ref{lem:notriangle} we have $\om_*(\mathfrak{c})=0$, $\om_*(\mathfrak{r}_0)=-m/2$ and $\om_*(\mathfrak{a})=n/2$. Then  Lemma~\ref{lem:omegaderivada} with the help of the Intermediate value theorem implies the existence of a unique 
 $\mathfrak{r}_1\in (\mathfrak{r}_0,\mathfrak{a})$ such that  $\om_*(\mathfrak{r}_1)=0$.
On the one hand,
by Lemma~\ref{lem:notriangle} we have 
$\omega_*(r)=0$ for $r\in (0,\mathfrak{c}]$, and hence the right-hand side of
\eqref{eq:formula} is equal to zero.
On the other hand, Lemma~\ref{lemTec2-1} and Item~(i) of Lemma~\ref{lem:boundedcases} imply that
the left-hand side of~\eqref{eq:formula} is equal to zero.
This yields~\eqref{eq:formula} for any $r\in (0, \mathfrak{c}]$.
 
In what follows, we assume that 
\begin{equation}\label{eq:asump1}
n\beta +m\alpha\quad  \textrm{ is not an integer multiple of } \pi.
\end{equation}
By Item (iv) of Proposition~\ref{eq:prop}, \eqref{eq:asump1} is equivalent to
$2P_*\not\in \mathbb{Z}$.
We note that $P_*\not\in \mathbb{Z}$.
We continue with the proof of~\eqref{eq:formula} for $r>\mathfrak{c}$.
The analysis 
is divided in the following two sub-cases.

Case~A.1:
Assume that the numbers
$P_*-\om_*(r)$ and $P_*+\om_*(r)$ are not integers for some $r\in \mathbb{T}$. 
Assume that $r\in J_1:=(\mathfrak{c},\mathfrak{b}]$.
By \eqref{eq:asump1},
Proposition~\ref{prop:pomega}, Proposition~\ref{prop:vuelta} and Lemma~\ref{lem:omegaderivada} we have
\[
\Card(\mathcal{Z}_{f}(r))=\Card_{\mathbb{Z}}(P_*-\om_*(0,r))+\Card_{\mathbb{Z}}(P_*+\om_*(0,r)).
\]
In particular, since $2P_*\not \in \mathbb{Z}$, for $r=\mathfrak{b}$ 
Lemma~\ref{eq:casogenerico09}  yields that $f$ does not have any root of modulus $\mathfrak{b}$
and hence 
$\Card(\mathcal{Z}_{f}(\mathfrak{b}))=m$.

Next, we assume that $r\in J_2:=(\mathfrak{b},\mathfrak{r}_1)$. We observe that
\begin{equation} 
\begin{split}
\Card_{\mathbb{Z}}(P_*-\om_*(0,r))&=\Card_{\mathbb{Z}}(P_*-\om_*(0,\mathfrak{b}])+ \Card_{\mathbb{Z}}(P_*-\om_*(\mathfrak{b},r)),\\
\Card_{\mathbb{Z}}(P_*+\om_*(0,r))&=\Card_{\mathbb{Z}}(P_*+\om_*(0,\mathfrak{b}])+ \Card_{\mathbb{Z}}(P_*+\om_*(\mathfrak{b},r)),
\end{split}
\end{equation}
which implies
\begin{equation} 
\begin{split}
&\Card_{\mathbb{Z}}(P_*-\om_*(0,r))
+\Card_{\mathbb{Z}}(P_*+\om_*(0,r))=
\Card_{\mathbb{Z}}(P_*-\om_*(0,\mathfrak{b}])\\
&\qquad\,+
\Card_{\mathbb{Z}}(P_*+\om_*(0,\mathfrak{b}])
+
\Card_{\mathbb{Z}}(P_*-\om_*(\mathfrak{b},r))+
\Card_{\mathbb{Z}}(P_*+\om_*(\mathfrak{b},r))\\
&\quad=m+\Card_{\mathbb{Z}}(P_*-\om_*(\mathfrak{b},r))+
\Card_{\mathbb{Z}}(P_*+\om_*(\mathfrak{b},r)).
\end{split}
\end{equation}
In particular, by  Lemma~\ref{eq:casogenerico09} for $r=\mathfrak{r}_1$ we obtain 
$\Card(\mathcal{Z}_{f}(\mathfrak{r}_1))=m+m=2m$.

Finally, we assume that $r> \mathfrak{r}_1$.
Similar reasoning implies
\begin{equation} 
\begin{split}
&\Card_{\mathbb{Z}}(P_*-\om_*(0,r))
+\Card_{\mathbb{Z}}(P_*+\om_*(0,r))\\
&\qquad=2m+\Card_{\mathbb{Z}}(P_*-\om_*(\mathfrak{r}_1,r))+
\Card_{\mathbb{Z}}(P_*+\om_*(\mathfrak{r}_1,r)).
\end{split}
\end{equation}
In particular, Lemma~\ref{eq:casogenerico09} for $r=\mathfrak{a}$ we obtain 
$\Card(\mathcal{Z}_{f}(\mathfrak{a}))=2m+n$.
 
Case~A.2: 
Assume that the numbers
 $P_*-\om_*(r)$ or $P_*+\om_*(r)$ are integers for some $r\in \mathbb{T}$. We stress that
$P_*-\om_*(r)$ and $P_*+\om_*(r)$ cannot be both integers.
Without loss of generality, one can assume that $P_*+\om_*(r)$ is an integer. 
Since the number of roots for any harmonic trinomial is finite,
for $0<\tilde{r}<r$ and $\tilde{r}$ sufficiently close to $r$ by continuity we obtain that 
$P_*-\om_*(\tilde{r})$ and $P_*+\om_*(\tilde{r})$ are not integers,
\begin{equation}
\begin{split}
\Card_{\mathbb{Z}}(P_*-\om_*(0,r))&=\Card_{\mathbb{Z}}(P_*-\om_*(0,\tilde{r})),\\
\Card_{\mathbb{Z}}(P_*+\om_*(0,r))&=\Card_{\mathbb{Z}}(P_*+\om_*(0,\tilde{r})).
\end{split}
\end{equation}
Hence Case~A.1 implies
$
\Card(\mathcal{Z}_{f}(r))=\Card(\mathcal{Z}_{f}(\tilde{r}))
$ and we conclude~\eqref{eq:formula}.
Case~A.1 and Case~A.2 complete the proof for the Case~A under~\eqref{eq:asump1}.

In the sequel, we assume that
\begin{equation}\label{eq:asump2}
n\beta +m\alpha\quad  \textrm{ is an integer multiple of } \pi.
\end{equation}
We analyze the following two cases.

Case~B.1: If $P_*-\om_*(r)$ and $P_*+\om_*(r)$ are not integers for some $r\in \mathbb{T}$, then the proof is similar to the Case~A.1.

Case~B.2:
If $P_*-\om_*(r)$ or $P_*+\om_*(r)$ are integers for some $r\in \mathbb{T}$, then Item~(i) and Item~(ii) of Proposition~\ref{eq:prop} imply that $P_*-\om_*(r)$ and $P_*+\om_*(r)$ are integers. 
Recall that $n\alpha+m \beta$ 
satisfies~\eqref{eq:asump2}.
By perturbation we choose $\widetilde{\alpha}$ and $\widetilde{\beta}$ such that 
$n\widetilde{\alpha}+m\widetilde{\beta}$ satisfies~\eqref{eq:asump1}. For such choice, 
the associated harmonic trinomial $g$ is given by
$
g(\zeta)=|a|e^{\ii \widetilde{\alpha}}\zeta^n+
|b|e^{\ii \widetilde{\beta}}\overline{\zeta}^m+c,
$ $\zeta\in \mathbb{C}$
and it has corresponding pivots 
\[
P^g_*=\frac{n(\widetilde{\beta}-\pi)+m(\widetilde{\alpha}-\pi)}{2\pi}
\quad \textrm{ and }\quad
\omega^g_*(r)= \frac{n\om_1-m\om_2}{2\pi}.
\]
We stress that $\omega^g_*(r)=\omega_*(r)$ which is defined in~\eqref{eq:pivotaltes}.
In addition, such $\widetilde{\alpha}$ and $\widetilde{\beta}$ can be chosen to satisfy $0<P^g_*-P_*<1/4$, and  $P^g_*-\omega_*(r)$, $P^g_*+\omega_*(r)$ are not integers.
The preceding choice implies
\begin{equation}\label{eq:choice}
\begin{split}
\Card(\mathcal{Z}_f(r))&=\Card(\mathcal{Z}_g(r))-1,\\
\Card_{\mathbb{Z}}(P^g_*-\om_*(0,r))&=\Card_{\mathbb{Z}}(P_*-\om_*(0,r))+1,\\
\Card_{\mathbb{Z}}(P^g_*+\om_*(0,r))&=\Card_{\mathbb{Z}}(P_*+\om_*(0,r)).
\end{split}
\end{equation}
Since $n\widetilde{\alpha}+m\widetilde{\beta}$ satisfies~\eqref{eq:asump1}, Case~(I) yields
\[
\Card(\mathcal{Z}_{g}(r))=\Card_{\mathbb{Z}}(P^g_*-\om_*(0,r))+\Card_{\mathbb{Z}}(P^g_*+\om_*(0,r)).
\]
The preceding equality with the help 
of~\eqref{eq:choice} 
implies~\eqref{eq:formula}.
Case~B.1 and Case~B.2 complete the proof for the Case~B under~\eqref{eq:asump2}.

The proofs of Case~(II) and Case~(III)
follow step by step from the Case~(I).
The proof of Theorem~\ref{th:bohlharm} is finished.
\end{proof}

\appendix
{
\section{Continuity and trigonometric equations associated to triangles}\label{append}
This section contains useful properties that help us to make this paper more fluid. Since all proofs are straightforward, we left the details to the interested reader.
\begin{lemma}[Continuity]\label{eq:continuity}
Let $\Delta$ be a triangle (it may be degenerate) with side lengths $\ell_1$, $\ell_2$ and $\ell_3$.
Let $\om_1$, $\om_2$ and $\om_3$
be the interior angles opposite to the sides with lengths $\ell_1$, $\ell_2$ and $\ell_3$, respectively. Then $\om_1$, $\om_2$ and $\om_3$ are uniquely determined by  $\ell_1$, $\ell_2$ and $\ell_3$.
Moreover, they are continuous functions of them.
\end{lemma}
\begin{lemma}[P.~558 Equation (3) of \cite{Bohl}]\label{triangulo}
Let $\Delta$ be a triangle (it may be degenerate) with side lengths $\ell_1$, $\ell_2$ and $\ell_3$.
Let $\om_1$, $\om_2$ and $\om_3$
be the interior angles opposite to the sides with lengths $\ell_1$, $\ell_2$ and $\ell_3$, respectively. Then the following trigonometric equation holds true
$\ell_1 e^{-\ii \om_2}+\ell_2 e^{\ii \om_1}-\ell_3=0$.
\end{lemma}
\begin{lemma}[P.~558 Equation (4)-(6) of \cite{Bohl}]\label{prop:congru}
Let $\Delta$ be a triangle (it may be degenerate) with side lengths $\ell_1$, $\ell_2$ and $\ell_3$. Then the real numbers $\phi$ and $\psi$ that are solutions of the complex trigonometric equation
$
\ell_1 e^{\ii\phi}+\ell_2 e^{\ii\psi}+\ell_3=0
$
satisfy the relations
$\phi\equiv \pi-\om_2,\, \psi\equiv \pi+\om_1$
 and 
$\phi\equiv -\pi+\om_2,\, \psi\equiv -\pi-\om_1$,
where $A\equiv B$ means that $A-B= 2\pi k$ for some $k\in \mathbb{Z}$.
\end{lemma}
}

\section*{Acknowledgments}
G. Barrera would like to express his gratitude to University of Helsinki, Department of Mathematics and Statistics, for all the facilities used along the realization of this work.
All authors are greatly indebted to Mar\'{\i}a F. Delf\'{\i}n Ares de Parga (University of Oregon, Eugene, USA) and Michael A. H\"ogele 
(Universidad de los Andes, Bogot\'a, Colombia) for their support on the translation of \cite{Bohl}.
The authors are grateful to the reviewer for the thorough examination of the paper, which has lead to a significant improvement.

\section*{Declarations}
\subsection*{Funding} 
The research of G. Barrera has been supported by the Academy of Finland, via 
the Matter and Materials Profi4 University Profiling Action, an Academy project (project No. 339228)
and the Finnish Centre of Excellence in Randomness and STructures (project No. 346306). The research of W. Barrera and J. P. Navarrete has been supported by the CONACYT, Proyecto Ciencia de Frontera 2019--21100 via Faculty of Mathematics, UADY.
\subsection*{Conflict of interests} The authors declare that they have no conflict of interest.
\subsection*{Authors' contributions}
All authors have contributed equally to the paper.

\bibliographystyle{amsplain}

\begin{thebibliography}{40}
\addcontentsline{}{}{}
%B
\bibitem{Bohl} 
Bohl, P.: 
Zur theorie der trinomischen gleichungen. 
\textit{Math. Ann.} \textbf{65}, no. 4, (1908), 556--566.

\bibitem{Botta} 
Botta, V., da Silva, J.:
On the behavior of roots of trinomial equations.
\textit{Acta Math. Hungar.} \textbf{157}, no. 1, (2019), 54--62.

\bibitem{BRILLESLYPER} 
Brilleslyper, M., Brooks, J., Dorff, M., Howell, R. \& Schaubroeck, L.: 
Zeros of a one-parameter family of harmonic trinomials.
\textit{Proc. Amer. Math. Soc. Ser. B} \textbf{7}, (2020), 82--90.

\bibitem{BrilleslyperSchaubroeck2018} 
Brilleslyper, M. \& Schaubroeck, L.:
Counting interior roots of trinomials.
\textit{Math. Mag.} \textbf{91}, no. 2, (2018), 142--150.

\bibitem{BrilleslyperSchaubroeck2014} 
Brilleslyper, M. \& Schaubroeck, L.:
Locating unimodular roots.
\textit{College Math. J.} \textbf{45}, no. 3, (2014), 162--168.

\bibitem{Brooks} 
Brooks, J., Dorff, M., Hudson, A., Pitts, E., Whiffen, C. \&  Woodall, A.: 
Zeros of a family of complex-valued harmonic trinomials.
\textit{Bull. Malays. Math. Sci. Soc.} (2022). \url{https://doi.org/10.1007/s40840-021-01230-8}

\bibitem{Bshouty}
Bshouty, D., Hengartner, W. \& Suez, T.: The exact bound on the number of zeros of harmonic polynomials. 
\textit{J. Anal. Math.} \textbf{67}, (1995), 207--218.

%F
\bibitem{Fell} 
Fell, H.:
The geometry of zeros of trinomial equations.
\textit{Rend. Circ. Mat. Palermo} \textbf{29}, no. 2, (1980), 303--336.

%G
\bibitem{Geyer} 
Geyer, L.:
Sharp bounds for the valence of certain harmonic polynomials.
\textit{Proc. Amer. Math. Soc.} \textbf{136}, no. 2, (2008), 549--555.

%H
\bibitem{Hauenstein}
Hauenstein, J., Lerario, A., Lundberg, E. \& Mehta, D.:
Experiments on the zeros of harmonic polynomials using certified counting. 
\textit{Exp. Math.} \textbf{24}, no. 2, (2015), 133--141.

\bibitem{Howell} 
Howell, R. \& Kyle, D.: 
Locating trinomial zeros.
\textit{Involve} \textbf{11}, no. 4, (2018), 711--720.

%K
\bibitem{Khavinson} 
Khavinson, D., Lee, S, \& Saez, A.:
Zeros of harmonic polynomials, critical
lemniscates, and caustics.
\textit{Complex Anal. Synerg.} \textbf{4}, no. 2, (2018), 1--20.

\bibitem{KhavinsonSwiatek} 
Khavinson, D. \&  \'Swia\k{a}tek, G.:
On the number of zeros of certain harmonic polynomials.
\textit{Proc. Amer. Math. Soc.} \textbf{131}, no. 2, (2003), 409--414.

\bibitem{Kuruklis} 
Kuruklis, S.:
The asymptotic stability of $x_{n+1} - ax_n+b x_{n- k}=0$.
\textit{J. Math. Anal. Appl.} \textbf{188}, no. 3, (1994), 719--731.

%L
\bibitem{Lee} 
Lee, S.-Y., Lerario, A. \& Lundberg, E.:
Remarks on Wilmshurst's theorem.
\textit{Indiana Univ. Math. J.} \textbf{64}, no. 4, (2015), 1153--1167.

%M 
\bibitem{Mardenbook} 
Marden, M.:
\textit{Geometry of polynomials}.
Amer. Math. Soc.,
Mathematical Surveys and Monographs
\textbf{3}, (1949), 1--243.

\bibitem{Melman} 
Melman, A.:
Geometry of polynomials.
Pacific J. Math. \textbf{259}, no. 1, (2012), 141--159.

%N
\bibitem{Nekrassoff} 
Nekrassoff, P.:
Ueber trinomische Gleichungen.
Math. Ann. \textbf{29}, (1887), no. 3, 413--430.

%S
\bibitem{Szabo}
Szab\'o, P.:
On the roots of the trinomial equation. 
\textit{Cent. Eur. J. Oper. Res.} \textbf{18}, no. 1, (2010), 97--104.
 
%W
\bibitem{Wilmshurst} 
Wilmshurst, A.: 
The valence of harmonic polynomials.
\textit{Proc. Amer. Math. Soc.} \textbf{126}, no. 7, (1998), 2077--2081.

\end{thebibliography}

\end{document}